\documentclass[11pt,a4paper]{article}
\usepackage{epsf,epsfig}
\usepackage{amsmath,amsfonts,amsgen,amstext,amsbsy,amsopn,amsthm,amssymb}
\usepackage{enumerate}
\usepackage{hhline}
\usepackage{multirow}
\usepackage{multicol}
\usepackage{booktabs}
\usepackage{lineno}
\usepackage{bm}
\usepackage{graphics}
\usepackage{latexsym}
\usepackage{bbding}
\usepackage{indentfirst}
\usepackage{graphicx}
\usepackage{color}
\usepackage[colorlinks=true,anchorcolor=blue,filecolor=blue,linkcolor=red,urlcolor=blue,citecolor=blue]{hyperref}
\hypersetup{pdftitle={Spectral Extremal Graphs for Even Factors and Connected Even Factors},pdfauthor={},hypertexnames=false}
\usepackage{float}
\usepackage{tikz}
\usetikzlibrary{calc,positioning,arrows.meta,fit,backgrounds,decorations.pathreplacing,shapes.geometric}
\allowdisplaybreaks[4]
\usepackage{geometry}
\newcommand{\mb}[1]{\boldsymbol{#1}}
\newcommand{\n}{\noindent}

\newtheorem{thm}{Theorem}[section]

\newtheorem{lemma}[thm]{Lemma}
\newtheorem{claim}{Claim}[section]

\newtheorem{prop}[thm]{Proposition}

\renewcommand{\leq}{\leqslant}
\renewcommand{\le}{\leqslant}
\renewcommand{\geq}{\geqslant}
\renewcommand{\ge}{\geqslant}

\title{Spectral extremal graphs for even factors}
\author{%
Zeyuan Wu\thanks{School of Mathematics and
Statistics, Minnan Normal University, Zhangzhou, Fujian, 363000,
China. Email: \url{zywu576@163.com}.}
\and
Hongzhang Chen\thanks{School of Mathematics and Statistics, Gansu
Center for Applied Mathematics, Lanzhou University, Lanzhou, Gansu,
730000, China. Email: \url{mnhzchern@gmail.com}.}
\and Xinting Shi\thanks{School of Mathematics and
Statistics, Minnan Normal University, Zhangzhou, Fujian, 363000,
China. Email: \url{xtshii@163.com}.}
\and Jianxi
Li\thanks{Corresponding author. School of Mathematics and
Statistics, Minnan Normal University, Zhangzhou, Fujian, 363000,
China. Email: \url{ptjxli@hotmail.com}. Partially supported by the
NSF of Fujian Province (No.\ 2026J001968).}%
}

\date{\today}

\begin{document}
\maketitle

\begin{abstract}
An even factor of a graph $G$ is a spanning subgraph in which every vertex has positive even degree. 
It is known that the minimum degree $\delta(G)\ge 2$ is a trivial necessary condition for $G$ to have an even factor.  Recent spectral results for the existence of even factors used the certain complete-join graphs as exceptional extremal graphs.  However, these graphs already contain $2$-factors and therefore are not genuine obstructions.  This observation leads to the
natural problem of determining the true sharp spectral threshold when the minimum
degree is given.  In this paper, we provide tight adjacency spectral radius conditions for a connected graph 
to contain an even factor, and characterize all extremal graphs, respectively. We also study the stronger 
requirement of a connected even factor, equivalently a spanning connected Eulerian
subgraph.  For this property, we also establish the corresponding sharp adjacency spectral radius condition and determine the unique extremal graph.
\end{abstract}



\section{Introduction}

All graphs considered in this paper are finite, simple, and undirected.
Let $G=(V(G),E(G))$ be a graph.  We write $|V(G)|=n$ for its order and
$e(G)=|E(G)|$ for its size.  For a vertex $v\in V(G)$, we denote its degree by
$d_G(v)$, or simply by $d(v)$ when the graph is clear from the context.  We denote
the minimum degree of $G$ by $\delta(G)$, or simply by $\delta$.  For a subset
$S\subseteq V(G)$, we use $G[S]$ and $G-S$ to denote the subgraphs of $G$ induced
by $S$ and by $V(G)\setminus S$, respectively.  If $G_1$ and $G_2$ are
vertex-disjoint graphs, then $G_1\cup G_2$ denotes their disjoint union, and
$G_1\vee G_2$ denotes the graph obtained from $G_1\cup G_2$ by joining every
vertex of $G_1$ to every vertex of $G_2$.  If $X,Y\subseteq V(G)$ are disjoint,
then $e_G(X,Y)$ denotes the number of edges joining $X$ and $Y$.

Let $A(G)$ be the adjacency matrix of a graph $G$.  Since $A(G)$ is real symmetric, all
its eigenvalues are real.  We order them as
$\lambda_{1}(A(G))\geq \lambda_{2}(A(G))\geq \cdots \geq
\lambda_{n}(A(G))$.  The largest eigenvalue $\lambda_{1}(A(G))$, denoted by
$\rho(G)$, is called the \emph{spectral radius} of $G$.  The matrix $A(G)$ is
nonnegative, and it is irreducible if and only if $G$ is connected.  Hence, when
$G$ is connected, the Perron--Frobenius theorem implies that $\rho(G)$ is
positive and simple, and that $G$ has a unique positive unit eigenvector
${\mb x}=(x_1,x_2,\ldots,x_n)^{T}$ associated with $\rho(G)$.  We call this
vector the Perron vector of $G$.

A spanning subgraph $F$ of a graph $G$ is called an \textit{even factor} if $d_F(v)$ is a
positive even integer for every $v\in V(G)$.  Thus a 2-factor is a special even
factor.  Even factors belong to the broader theory of graph factors, which seeks
spanning subgraphs satisfying prescribed degree constraints.  Recently, spectral
sufficient conditions for even factors have been investigated from several
viewpoints; see, for example, \cite{liu,JiaLuZhou,zhou}.  In such results, a
natural extremal candidate is often the complete-join graph
\[
 J_{n,\delta}=K_\delta\vee
 \bigl(K_{n-2\delta+1}\cup(\delta-1)K_1\bigr).
\]
However, this graph is not a genuine obstruction to the existence of even
factors.  Indeed, when $n\ge 2\delta+2$, the clique $K_{n-2\delta+1}$ has a
Hamilton cycle, while the remaining vertices can be covered by another
cycle using the $K_\delta$ part and the $\delta-1$ isolated vertices before
the join. Hence $J_{n,\delta}$ contains a 2-factor, and consequently it
already contains an even factor.  This shows that the extremal graph in a
sharp adjacency-spectral theorem must be different.

The purpose of this paper is to give corrected adjacency-spectral extremal
statements for even factors in connected graphs with prescribed minimum
degree.  Note that \(\delta\geq 2\) is a trivial necessary condition for a graph 
to have an even factor. 
We first consider $\delta\ge3$ and $n\ge\delta^2+\delta+1$.  The graph $\mathcal A_{n,\delta}$ (shown in Fig.~\ref{fig:obstructions}) is constructed from one vertex $u$, one clique $C_0\cong K_{n-\delta^2}$, and $\delta-1$ further cliques
$C_1,\ldots,C_{\delta-1}\cong K_{\delta+1}$. Choose a root
$a_i\in V(C_i)$ for each $i=0,1,\ldots,\delta-1$, and add exactly the
$\delta$ edges $ua_i$. Then $|V(\mathcal A_{n,\delta})|
 =1+n-\delta^2+(\delta-1)(\delta+1)=n$
and $\delta(\mathcal A_{n,\delta})=\delta$.

\begin{thm}\label{thm:delta-ge3}
For $\delta\ge3$ and $n\ge\delta^2+\delta+1$, let $G$ be a connected
graph of order $n$ with $\delta(G)\ge\delta$. If
$\rho(G)\ge\rho(\mathcal A_{n,\delta})$, then $G$ has an even factor unless
$G\cong\mathcal A_{n,\delta}$.
\end{thm}

For the case $\delta=2$, a different construction is needed. Let $n\ge5$.  The graph $\mathcal B_n$ (shown in Fig.~\ref{fig:obstructions}) is obtained from a clique
$C\cong K_{n-4}$ with a distinguished root $a\in V(C)$ by adding four vertices
$b, x_1, x_2, x_3$ and the six edges $ax_i$ and $bx_i$ for
$i=1,2,3$.

\begin{thm}\label{thm:delta2}
For $n\ge13$, let $G$ be a connected graph of order $n$ with
$\delta(G)=2$. If $\rho(G)\ge\rho(\mathcal B_n)$, then $G$ has an even factor
unless $G\cong\mathcal B_n$.
\end{thm}

We also consider the connected version of even factors.  A \emph{connected even factor} of a graph is a spanning connected subgraph in which every vertex has even degree.  Since all graphs considered here have order at least two, the connectedness condition already forces every vertex to have positive degree.  Equivalently, a graph has a connected even factor if and only if it contains a spanning connected Eulerian subgraph; in the usual terminology, such a graph is supereulerian.

For $\delta\geq2$ and $n\geq2\delta+2$, let $\mathcal C_{n,\delta}$ (shown in Fig.~\ref{fig:obstructions}) be the graph obtained from two disjoint cliques $K_{\delta+1}$ and $K_{n-\delta-1}$ by choosing one vertex in each clique and joining the two chosen vertices by exactly one edge.  This joining edge is a bridge.  Hence $\mathcal C_{n,\delta}$ has no connected even factor.  Our connected-even-factor result is as follows.

\begin{thm}\label{thm:connected-even}
Let $G$ be a connected graph of order $n$ with $\delta(G)\geq\delta\geq2$.  Assume that $n\geq8$ if $\delta\in\{2,3\}$, and that $n\geq2\delta+2$ if $\delta\geq4$.  If
\[
        \rho(G)\geq\rho(\mathcal C_{n,\delta}),
\]
then $G$ has a connected even factor unless $G\cong\mathcal C_{n,\delta}$.
\end{thm}

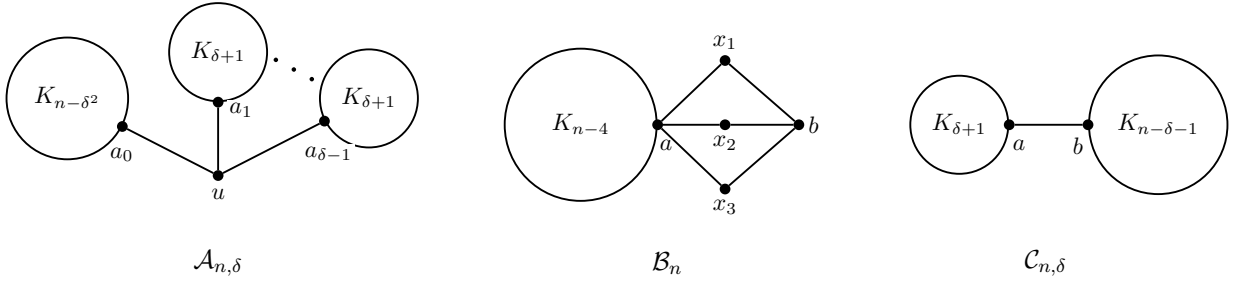
\begin{figure}[htbp]
\centering
\resizebox{\textwidth}{!}{
\begin{tikzpicture}
[scale=0.95,
every node/.style={font=\small},
dot/.style={circle,fill=black,inner sep=1.6pt,outer sep=0pt},
clique/.style={
circle,
draw=black,
thick,
minimum size=1.45cm,
align=center},
bigclique/.style={
circle,
draw=black,
thick,
minimum size=1.8cm,
align=center},
labelnode/.style={
inner sep=1pt,
fill=white},
edge/.style={thick}
]

\begin{scope}[xshift=-7.0cm]
\node[font=\bfseries] at (0,-2.15) {\(\mathcal A_{n,\delta}\)};

\begin{scope}[yshift=-0.79cm]
\node[dot] (u) at (0,0) {};
\node[labelnode] at (0,-0.28) {\(u\)};

\node[bigclique] (C0) at (-2.35,1.20) {\(K_{n-\delta^2}\)};
\node[dot] (a0) at (C0.-27) {};
\node[labelnode] at (-1.5,0.33) {\(a_0\)};
\draw[edge] (u)--(a0);

\node[clique] (C1) at (0,1.92) {\(K_{\delta+1}\)};
\node[dot] (a1) at (C1.-90) {};
\node[labelnode] at (0.36,1.02) {\(a_1\)};
\draw[edge] (u)--(a1);

\node[clique] (Cd) at (2.35,1.20) {\(K_{\delta+1}\)};
\node[dot] (ad) at (Cd.207) {};
\node[labelnode] at (1.72,0.33) {\(a_{\delta-1}\)};
\draw[edge] (u)--(ad);

\foreach \x/\y in {0.9/1.8,1.18/1.65,1.46/1.5}
\fill (\x,\y) circle (1.0pt);
\end{scope}
\end{scope}

\begin{scope}[xshift=0cm]
\node[font=\bfseries] at (0,-2.15) {\(\mathcal B_n\)};

\node[bigclique, minimum size=2.25cm] (C) at (-1.35,0) {\(K_{n-4}\)};
\node[dot] (a) at (C.0) {};
\node[labelnode] at (-0.02,-0.32) {\(a\)};

\node[dot] (b) at (2.05,0) {};
\node[labelnode] at (2.28,0) {\(b\)};

\node[dot] (x1) at (0.9,1.0) {};
\node[labelnode] at (0.9,1.28) {\(x_1\)};

\node[dot] (x2) at (0.9,0) {};
\node[labelnode] at (0.9,-0.28) {\(x_2\)};

\node[dot] (x3) at (0.9,-1.0) {};
\node[labelnode] at (0.9,-1.28) {\(x_3\)};

\draw[edge] (a)--(x1)--(b);
\draw[edge] (a)--(x2)--(b);
\draw[edge] (a)--(x3)--(b);
\end{scope}

\begin{scope}[xshift=5.9cm]
\node[font=\bfseries] at (0,-2.15) {\(\mathcal C_{n,\delta}\)};

\node[clique] (L) at (-1.35,0) {\(K_{\delta+1}\)};
\node[dot] (ca) at (L.0) {};
\node[labelnode] at (-0.42,-0.32) {\(a\)};

\node[bigclique, minimum size=2.05cm] (R) at (1.75,0) {\(K_{n-\delta-1}\)};
\node[dot] (cb) at (R.180) {};
\node[labelnode] at (0.5,-0.32) {\(b\)};

\draw[edge] (ca)--(cb);
\end{scope}
\end{tikzpicture}
}
\caption{The graphs \(\mathcal{A}_{n,\delta}\), \(\mathcal{B}_n\) and \(\mathcal{C}_{n,\delta}\).}
\label{fig:obstructions}
\end{figure}

Proposition~\ref{prop:candidates-no-even} verifies that $\mathcal A_{n,\delta}$ and $\mathcal B_n$ have no even factor, while Proposition~\ref{prop:C-no-connected-even} verifies that $\mathcal C_{n,\delta}$ has no connected even factor.

\section{Some properties of extremal graphs}

In this section, we collect the elementary properties of the extremal graphs.  We first show that the complete-join \(J_{n,\delta}\) is not a genuine obstruction to the existence of an even factor.  We then verify that \(\mathcal{A}_{n,\delta}\) and \(\mathcal{B}_n\) have no even factor, and finally verify the basic obstruction for the connected-even-factor extremal graph \(\mathcal C_{n,\delta}\). 

\begin{prop}\label{prop:false-exception}
For $\delta\geq 2$ and $n\geq 2\delta+2$, the graph
$J_{n,\delta}=K_\delta\vee(K_{n-2\delta+1}\cup(\delta-1)K_1)$ has a
2-factor.  In particular, $J_{n,\delta}$ is not a genuine obstruction to the
existence of an even factor.
\end{prop}

\begin{proof}
Let $S=\{u_1,\ldots,u_\delta\}$ be the vertex set of the $K_\delta$ part, let
$I=\{x_1,\ldots,x_{\delta-1}\}$ be the vertices that are isolated before the
join, and let $W$ be the vertex set of the clique $K_{n-2\delta+1}$.  Since
$|W|=n-2\delta+1\geq 3$, the clique on $W$ has a Hamilton cycle, denoted by
$C_W$.

It remains to cover $S\cup I$ by one cycle.  Consider
\[
C_0=x_1u_1x_2u_2\cdots x_{\delta-1}u_{\delta-1}u_\delta x_1.
\]
When $\delta=2$, the cycle is simply $x_1u_1u_2x_1$.
The disjoint union $C_0\cup C_W$ is a spanning 2-regular subgraph of $J_{n,\delta}$.  Hence $J_{n,\delta}$ has a 2-factor.
\end{proof}

\begin{prop}\label{prop:candidates-no-even}
The graph $\mathcal A_{n,\delta}$ has no even factor for every
$\delta\geq 3$ and $n\geq \delta^2+\delta+1$.  The graph $\mathcal B_n$
has no even factor for every $n\geq 5$.
\end{prop}

\begin{proof}
We first consider $\mathcal A_{n,\delta}$.  Let
$C_0,C_1,\ldots,C_{\delta-1}$ be its clique branches, and let
$r_i$ be the root of $C_i$ adjacent to the central vertex $u$.  Thus
\[
\partial_{\mathcal A_{n,\delta}}V(C_i)=\{ur_i\},
\qquad 0\leq i\leq \delta-1,
\]
where $\partial_G X$ denotes the set of edges of $G$ with exactly one
end in $X$.

Suppose that $F$ is a spanning subgraph of $\mathcal A_{n,\delta}$ such
that every vertex has even degree in $F$.  For every $X\subseteq
V(\mathcal A_{n,\delta})$, we have
\[
\sum_{x\in X}d_F(x)
=
2e_F(X)+e_F(X,V(\mathcal A_{n,\delta})\setminus X).
\]
Hence
\[
e_F(X,V(\mathcal A_{n,\delta})\setminus X)
\equiv
\sum_{x\in X}d_F(x)
\equiv 0 \pmod 2.
\]
Applying this with $X=V(C_i)$ gives
\[
\mathbf 1_{\{ur_i\in E(F)\}}
=
e_F(V(C_i),V(\mathcal A_{n,\delta})\setminus V(C_i))
\equiv 0 \pmod 2.
\]
Since the left-hand side is either $0$ or $1$, it follows that
$\mathbf 1_{\{ur_i\in E(F)\}}=0$ for every $0\leq i\leq \delta-1$.
Therefore
\[
d_F(u)=\sum_{i=0}^{\delta-1}\mathbf 1_{\{ur_i\in E(F)\}}=0.
\]
Thus every even-degree spanning subgraph of $\mathcal A_{n,\delta}$
isolates $u$.  In particular, no such subgraph can be an even factor,
since an even factor requires every vertex to have positive even
degree.

We next consider $\mathcal B_n$.  With the notation used in the
construction of $\mathcal B_n$, let $a$ be the distinguished vertex in
the clique $K_{n-4}$, let $b$ be the remaining vertex of the $K_{2,3}$
part, and let $x_1,x_2,x_3$ be the three vertices of degree two.  Thus
\[
N_{\mathcal B_n}(x_i)=\{a,b\}\quad (1\leq i\leq 3),
\qquad
N_{\mathcal B_n}(b)=\{x_1,x_2,x_3\}.
\]
Assume, to the contrary, that $F$ is an even factor of $\mathcal B_n$.
For $1\leq i\leq 3$, set
$\alpha_i=\mathbf 1_{\{ax_i\in E(F)\}}$ and $\beta_i=\mathbf 1_{\{bx_i\in E(F)\}}$.
Since $F$ is an even factor, $d_F(x_i)$ is positive and even.  On the
other hand, $d_{\mathcal B_n}(x_i)=2$, so
\[
2\geq d_F(x_i)=\alpha_i+\beta_i\in 2\mathbb Z_{>0}.
\]
Consequently, $\alpha_i+\beta_i=2$, and hence
$\alpha_i=\beta_i=1$ for every $1\leq i\leq 3$.
It follows that
\[
d_F(b)=\sum_{i=1}^3\mathbf 1_{\{bx_i\in E(F)\}}=
\sum_{i=1}^3\beta_i=3,
\]
which is not even.  This contradicts the definition of an even factor.
Therefore $\mathcal B_n$ has no even factor.
\end{proof}

\begin{prop}\label{prop:C-no-connected-even}
The graph $\mathcal C_{n,\delta}$ has no connected even factor for every $\delta\geq2$ and $n\geq2\delta+2$.
\end{prop}

\begin{proof}
Let $X$ be the vertex set of the clique $K_{\delta+1}$ in $\mathcal C_{n,\delta}$.  Then
$e_{\mathcal C_{n,\delta}}(X,V(\mathcal C_{n,\delta})\setminus X)=1$. 
Suppose that $F$ is a spanning connected subgraph of $\mathcal C_{n,\delta}$ such that every vertex has even degree in $F$.  Since $F$ is connected and $\emptyset\neq X\subset V(\mathcal C_{n,\delta})$, we have 
$e_F(X,V(\mathcal C_{n,\delta})\setminus X)\geq1$.
On the other hand,
\[
        \sum_{v\in X}d_F(v)=2e_F(X)+e_F(X,V(\mathcal C_{n,\delta})\setminus X).
\]
Therefore, we have 
$e_F(X,V(\mathcal C_{n,\delta})\setminus X)\equiv0\pmod2$. 
Since this number is at most $e_{\mathcal C_{n,\delta}}(X,V(\mathcal C_{n,\delta})\setminus X)=1$, it must be $0$, contradicting the connectedness of $F$.  Therefore $\mathcal C_{n,\delta}$ has no connected even factor.
\end{proof}

\section{Preliminaries}

Before proving Theorems~\ref{thm:delta-ge3} and \ref{thm:delta2}, we record the preliminary results that we need.

\begin{lemma}[\cite{Bapat2010}]\label{lem-1}
Let $G$ be a connected graph and let $H$ be a subgraph of $G$.  Then
$\rho(H)\leq \rho(G)$. Moreover, equality holds if and only if $H=G$.  In particular, if $H$ is a proper subgraph of the connected graph $G$, then $\rho(H)<\rho(G)$.
\end{lemma}

\begin{lemma}[\cite{AiImKimLeeZhang2024}]\label{lem-2}
Let $G,H_1,\ldots,H_k$ be graphs on the same vertex set, and suppose that $E(G)=\bigcup_{i=1}^k E(H_i)$.
Then
\[
        \rho(G)\leq \sum_{i=1}^k \rho(H_i).
\]
Consequently, if $E(G)=E(H)\cup E(F)$ and $E(H)\cap E(F)=\emptyset$, then \(\rho(G)\leq \rho(H)+\rho(F)\). 
\end{lemma}

\begin{lemma}[\cite{Bapat2010}]\label{lem-3}
Let $G$ be a bipartite graph with $e(G)$ edges.  Then \(\rho(G)\leq \sqrt{e(G)}\). 
\end{lemma}

\begin{lemma}\label{lem:hong-component}
Let $F$ be a graph with $q\geq1$ edges.  Then \(\rho(F)\leq \sqrt{2q-1}\). 
\end{lemma}

\begin{proof}
Let $F_0$ be a connected component of $F$ such that $\rho(F)=\rho(F_0)$, and write $n_0=|V(F_0)|$ and $q_0=e(F_0)$.  Since $F_0$ contains an edge, $n_0\geq2$ and $1\leq q_0\leq q$.  Hong's bound~\cite{Hong1988} for a connected graph of order $n_0$ and size $q_0$ gives
\[
        \rho(F_0)\leq\sqrt{2q_0-n_0+1}.
\]
As $n_0\geq2$, we have
$2q_0-n_0+1\leq2q_0-1\leq2q-1$. Therefore
\[
        \rho(F)=\rho(F_0)\leq\sqrt{2q_0-n_0+1}\leq\sqrt{2q-1},
\]
as desired.
\end{proof}

\begin{lemma}[\cite{WuXiaoHong2005}]\label{lem-4}
Let $G$ be a connected graph, and let $\mb x$ be the  Perron vector of $G$.  Suppose $uv\in E(G)$ and $uw\notin E(G)$.  Let $G'=G-uv+uw$. If $\mb{x}_w\ge \mb{x}_v$, then $\rho(G')>\rho(G)$.
\end{lemma}

Consider an $n\times n$ real symmetric matrix
$$
M=\left(\begin{array}{cccc}
M_{1,1} & M_{1,2} & \cdots & M_{1, m} \\
M_{2,1} & M_{2,2} & \cdots & M_{2, m} \\
\vdots & \vdots & \ddots & \vdots \\
M_{m, 1} & M_{m, 2} & \cdots & M_{m, m}
\end{array}\right)
$$
whose rows and columns are partitioned according to a partition
$X_{1},X_{2},\ldots,X_{m}$ of $\{1,2,\ldots,n\}$.
The {\it quotient matrix} $\mathcal{B}$ of $M$ is the $m\times m$ matrix whose
entries are the average row sums of the blocks $M_{i,j}$.  The partition is
{\it equitable} if each block $M_{i,j}$ has constant row sum.

\begin{lemma}[\cite{You}]\label{lem:equitable}
Let $M$ be a square matrix with an equitable partition $\pi$
and let $M_{\pi}$ be the corresponding quotient matrix.
Then every eigenvalue of $M_{\pi}$ is an eigenvalue of $M$.
Furthermore, if $M$ is nonnegative and $M_{\pi}$ is irreducible,
then the largest eigenvalues of $M$ and $M_{\pi}$ are equal.
\end{lemma}

The proof relies on the following barrier theorem for even factors.  We only need its necessary direction.

\begin{lemma}[\cite{LiZhang}]\label{lem:barrier}
Let $G$ be a connected graph with $\delta(G)\geq 2$. If $G$ has no even factor, 
then there exists a nonempty subset \(S\subseteq V(G)\) such that, 
letting \(C_1,\ldots,C_c\) be the components of \(G-S\), the following hold: 
\begin{enumerate}[(i)]
\item \(S\) is independent; 
\item \(c=\sum_{u\in S}d_G(u)-2|S|+2\); 
\item \(r_i=e_G(C_i,S)\) is odd for every \(i\in \{1,\ldots,c\}\); 
\item each vertex of \(S\) has at most one neighbour in each component \(C_i\). 
\end{enumerate}
\end{lemma}

Let \(S\) be a barrier set, and write \(s=|S|\). Let \(C_1,\ldots,C_c\) be the 
components of \(G-S\), and let \(r_i=e_G(C_i,S)\). Since \(S\) is independent, 
\(\sum_{u\in S}d_G(u)=\sum_{i=1}^{c}r_i\). Then Lemma~\ref{lem:barrier} implies that 
\(c=\sum_{u\in S}d_G(u)-2s+2\), and hence 
\begin{equation}\label{eq:surplus}
\sum_{i=1}^c(r_i-1)=\sum_{u\in S}d_G(u)-c=2s-2.
\end{equation}

We repeatedly use the following elementary consequences.

\begin{lemma}\label{lem:one-edge-size}
Let $G$ satisfy $\delta(G)\ge\delta\ge2$, let $S$ be a barrier set, and let $C$ be a component of $G-S$ with $e_G(C,S)=1$.  Then $|C|\ge\delta+1$.
\end{lemma}

\begin{proof}
Suppose that $|C|\leq \delta$.  Since $C$ sends exactly one edge to $S$, we have
\[
\delta |C|\leq \sum_{u\in V(C)}d_G(u)
       =2e(C)+1\leq |C|(|C|-1)+1.
\]
Thus $|C|(\delta-|C|+1)\leq 1$.  If $|C|=1$, then
$|C|(\delta-|C|+1)=\delta\geq2$; if $2\leq |C|\leq\delta$, then
$|C|(\delta-|C|+1)\geq2$.  Both alternatives contradict
$|C|(\delta-|C|+1)\leq1$.  Hence $|C|\geq\delta+1$.
\end{proof}

\begin{lemma}\label{lem:branch-compression}
Let $k\geq 2$, $m\geq 2$, and $N\geq km$.  For every
$k$-tuple $(a_1,\ldots,a_k)$ with $a_i\geq m$ and
$\sum_{i=1}^k a_i=N$, let $T(a_1,\ldots,a_k)$ be obtained
from a vertex $u$ and $k$ pairwise disjoint rooted cliques
$K_{a_1},\ldots,K_{a_k}$ by joining $u$ to the root of each
clique.  Then
\[
\rho\bigl(T(a_1,\ldots,a_k)\bigr)
\leq
\rho\bigl(T(N-m(k-1),m,\ldots,m)\bigr).
\]
Moreover, equality holds if and only if, up to a permutation of the branches, $(a_1,\ldots,a_k)=(N-m(k-1),m,\ldots,m)$.
\end{lemma}

\begin{proof}
For the $i$th branch, let $r_i$ be its root and let
$W_i$ be the set of its non-root vertices, so that
$|W_i|=a_i-1$.  We first prove the following compression step.

Assume, after relabelling, that $a_1\geq a_j>m$ for some
$j\geq 2$.  Let $T=T(a_1,\ldots,a_j,\ldots,a_k)$, and let
$T'$ be obtained from $T$ by replacing the pair of branch sizes
$(a_1,a_j)$ with $(a_1+1,a_j-1)$.  Equivalently, choose a vertex
$w\in W_j$ and define
\[
\Delta^-=\{wv:\ v\in \{r_j\}\cup (W_j\setminus\{w\})\},
\qquad
\Delta^+=\{wv:\ v\in \{r_1\}\cup W_1\}.
\]
Then
\[
E(T')=(E(T)\setminus \Delta^-)\cup \Delta^+.
\]

Let $\lambda=\rho(T)$, and let $\mb{x}$ be the positive Perron
vector of $T$.  Write
$x_u=x_0$, $x_{r_i}=y_i$ and $x_v=z_i$ for every $v\in W_i$.
The eigenvalue equations on the $i$-th branch are
\begin{equation}\label{e-1}
\lambda y_i=x_0+(a_i-1)z_i,
\end{equation}
\begin{equation}\label{e-2}
\lambda z_i=y_i+(a_i-2)z_i.
\end{equation}
Since $K_{a_i}$ is a proper subgraph of the connected graph $T$, Lemma~\ref{lem-1} gives $\lambda>a_i-1$.  Hence
$\lambda-a_i+2>0$, and by (\ref{e-2}),
\begin{equation}\label{e-3}
z_i=\frac{y_i}{\lambda-a_i+2}.
\end{equation}
Substituting (\ref{e-3}) into (\ref{e-1}) yields
\[
\left(\lambda-\frac{a_i-1}{\lambda-a_i+2}\right)y_i=x_0.
\]
Set $\theta_i=\lambda-\frac{a_i-1}{\lambda-a_i+2}$.
Since $x_0>0$ and $y_i>0$, we have $\theta_i>0$. Hence
$y_i=\frac{x_0}{\theta_i}$ and
$z_i=\frac{x_0}{(\lambda-a_i+2)\theta_i}$.

Now let $a_p\geq a_q$.  Then
\[
\frac{a_p-1}{\lambda-a_p+2}
-
\frac{a_q-1}{\lambda-a_q+2}
=
\frac{(a_p-a_q)(\lambda+1)}
{(\lambda-a_p+2)(\lambda-a_q+2)}
\geq 0.
\]
Thus $\theta_p\leq \theta_q$, and consequently
$y_p=\frac{x_0}{\theta_p}\geq \frac{x_0}{\theta_q}=y_q$.
Moreover,
\[
\frac{z_p}{z_q}
=
\frac{y_p}{y_q}\cdot
\frac{\lambda-a_q+2}{\lambda-a_p+2}
\geq 1,
\]
so $z_p\geq z_q$. In particular,
$y_1\geq y_j$ and $z_1\geq z_j$.

Since $x_w=z_j$, we have
\[
\begin{aligned}
\mb{x}^{T}(A(T')-A(T))\mb{x}
&=
2x_w
\left(
\sum_{v\in \{r_1\}\cup W_1}x_v
-
\sum_{v\in \{r_j\}\cup (W_j\setminus\{w\})}x_v
\right)\\
&=2z_j\left(y_1+(a_1-1)z_1-y_j-(a_j-2)z_j\right)\\
&\ge 2z_j(y_j+(a_1-1)z_j-y_j-(a_j-2)z_j)\\
&=2(a_1-a_j+1)z_{j}^2\\
&>0.
\end{aligned}
\]
Therefore,
\[
\rho(T')\geq\mb{x}^{T}A(T')\mb{x}
=\mb{x}^{T}A(T)\mb{x}+\mb{x}^{T}(A(T')-A(T))\mb{x}
>\lambda=\rho(T).
\]

Since there are only finitely many integer $k$-tuples
$(a_1,\ldots,a_k)$ satisfying $a_i\geq m$ and
$\sum_i a_i=N$, a maximizing tuple exists.  Let
$(b_1,\ldots,b_k)$ be such a tuple, ordered so that
$b_1\geq b_2\geq\cdots\geq b_k$.  If $b_2>m$, then the
compression step applied to the first and second branches
produces another admissible tuple with strictly larger spectral
radius, a contradiction.  Hence
$b_2=b_3=\cdots=b_k=m$. Since $\sum_i b_i=N$, we have
$b_1=N-m(k-1)$. Thus every maximizer has branch sizes
$(N-m(k-1),m,\ldots,m)$ up to a permutation of the branches.  This also gives the
uniqueness of the extremal graph up to isomorphism.
\end{proof}

\section{Proof of Theorem~\ref{thm:delta-ge3}}

\begin{proof}[Proof of Theorem~\ref{thm:delta-ge3}]
Suppose that \(\delta(G)\geq \delta\geq 3\) and that \(G\) has no even factor. It suffices to prove that
\begin{equation*}
\rho(G)\le\rho(\mathcal{A}_{n,\delta}),
\end{equation*}
with equality only when \(G\cong\mathcal{A}_{n,\delta}\).

Choose a barrier set \(S\) from Lemma~\ref{lem:barrier}. Let \(s=|S|\), and
let \(r_i=e_G(C_i,S)\), where \(C_1,\ldots,C_c\) are the components
of \(G-S\). Let \(p=|\{i:r_i=1\}|\). Since \(S\) is independent, we have
$\sum_{i=1}^c r_i=\sum_{u\in S}d_G(u)$.
Since all \(r_i\) are odd, Lemma~\ref{lem:barrier} implies that
\begin{equation}\label{equ-7}
2(c-p)\le\sum_{i=1}^c(r_i-1)=\sum_{u\in S}d_G(u)-c=2s-2.
\end{equation}
Thus
\begin{equation*}
p\ge c-s+1=\left(\sum_{u\in S}d_G(u)-2s+2\right)-s+1=\sum_{u\in S}d_G(u)-3s+3.
\end{equation*}
Note that \(\sum_{u\in S}d_G(u)\ge\delta s\). Then
\begin{equation*}
p\ge s(\delta-3)+3.
\end{equation*}

\begin{claim}
$p\ge \delta$.
\end{claim}

\begin{proof}
If $\delta=3$, then $p\ge s(\delta-3)+3\ge 3$.  If \(\delta\ge4\),
\begin{equation*}
p\ge s(\delta-3)+3=\delta+(s-1)(\delta-3)\ge\delta,
\end{equation*}
as desired.
\end{proof}

Select \(\delta\) components among those with \(r_i=1\), and relabel them as
\(C_1,\ldots,C_\delta\). Their unique edges to \(S\) are cut edges.
By Lemma~\ref{lem:one-edge-size}, we have
\begin{equation}\label{equ-8}
|C_i|\ge\delta+1\qquad 1\le i\le\delta.
\end{equation}

\begin{claim}
If \(s=1\) and \(\sum_{u\in S}d_G(u)=\delta\), then
\begin{equation*}
\rho(G)\le\rho(\mathcal{A}_{n,\delta}),
\end{equation*}
with equality only when \(G\cong\mathcal{A}_{n,\delta}\).
\end{claim}
\begin{proof}
Write \(S=\{u\}\). Since $c=\sum_{u\in S}d_G(u)-2s+2$, we have \(c=\delta\).
Moreover, $\sum_{i=1}^c r_i=\sum_{u\in S}d_G(u)=\delta$,
and each \(r_i\) is a positive odd integer. Hence
$r_1=r_2=\cdots=r_\delta=1$. Completing all components of \(G-u\) to cliques gives a graph
\(T(a_1,\ldots,a_\delta)\), where \(a_i\ge\delta+1\) and
\(\sum_{i=1}^{\delta}a_i=n-1\).
By Lemma~\ref{lem:branch-compression}, after relabelling if necessary, the
maximum spectral radius is attained only when
$a_1=n-\delta^2$ and $a_2=\cdots=a_\delta=\delta+1$.
Thus
\begin{equation*}
\rho(G)\le\rho(\mathcal{A}_{n,\delta}),
\end{equation*}
with equality only when \(G\cong\mathcal{A}_{n,\delta}\). This proves the claim.
\end{proof}

It remains to consider the case
\begin{equation*}
\left(s,\sum_{u\in S}d_G(u)\right)\ne(1,\delta).
\end{equation*}
Let \(E_0\) be the set of the \(\delta\) selected cut edges, let \(F\) be the graph with edge set \(E_0\), and let \(H=G-E_0\).
Then \(F\) is bipartite with \(\delta\) edges. By Lemma~\ref{lem-3}, we have
\begin{equation}\label{equ-9}
\rho(F)\le\sqrt{\delta}.
\end{equation}

\begin{claim}\label{c-4.2}
Every component of \(H\) has order at most \(n-\delta^2-\delta\).
\end{claim}
\begin{proof}
Let $R=V(G)\setminus\bigcup_{i=1}^{\delta}V(C_i)$.
By \eqref{equ-8}, we have
\begin{equation*}
\left|\bigcup_{i=1}^{\delta}V(C_i)\right|\ge\delta(\delta+1),
\end{equation*}
and hence
\begin{equation*}
|R|\le n-\delta(\delta+1)=n-\delta^2-\delta.
\end{equation*}
Therefore every component of \(H\) contained in \(R\) has order at most
\(n-\delta^2-\delta\).

Now let \(C_j\) be one of the selected components. Since
\(|C_i|\ge\delta+1\) for \(i\ne j\), it is enough to prove
\begin{equation}\label{equ-10}
|R|\ge\delta+1.
\end{equation}
Indeed, if \eqref{equ-10} holds, then
\begin{equation*}
n-|C_j|=|R|+\sum_{\substack{1\le i\le\delta\\ i\ne j}}|C_i|\ge(\delta+1)+(\delta-1)(\delta+1)=\delta^2+\delta,
\end{equation*}
and so
\begin{equation*}
|C_j|\le n-\delta^2-\delta.
\end{equation*}

It remains to prove \eqref{equ-10}. We distinguish the following three cases.
\begin{enumerate}[(i)]

\item \(s=1\).

Then the assumption
$\left(s,\sum_{u\in S}d_G(u)\right)\ne(1,\delta)$
implies $\sum_{u\in S}d_G(u)>\delta$.
Since \(S=\{u\}\) and each vertex of \(S\) has at most one neighbour in each
component of \(G-S\), we have \(r_i=1\) for every \(i\). Thus
$c=\sum_{u\in S}d_G(u)>\delta$.
Hence there exists a component \(C\) of \(G-S\), distinct from
\(C_1,\ldots,C_\delta\), such that \(e_G(C,S)=1\). By Lemma~\ref{lem:one-edge-size}, we have
\(|C|\ge\delta+1\). Since \(C\subseteq R\), we obtain
\begin{equation*}
|R|\ge s+|C|\ge1+(\delta+1)=\delta+2>\delta+1.
\end{equation*}

\item \(s\ge2\) and \(\delta\ge4\).

In this case, $p\ge\delta+(s-1)(\delta-3)\ge\delta+1$.
Hence there exists a component \(C\) of \(G-S\), distinct from
\(C_1,\ldots,C_\delta\), such that \(e_G(C,S)=1\). By Lemma~\ref{lem:one-edge-size}, we have
\(|C|\ge\delta+1\). Since \(C\subseteq R\), we obtain
\begin{equation*}
|R|\ge s+|C|\ge s+(\delta+1)>\delta+1.
\end{equation*}

\item \(\delta=3\) and \(s\ge2\).

If \(p\ge4\), then there exists a component \(C\) of \(G-S\), distinct from
\(C_1,C_2,C_3\), such that \(e_G(C,S)=1\). By Lemma~\ref{lem:one-edge-size}, \(|C|\ge4\).
Since \(C\subseteq R\),
\begin{equation*}
|R|\ge s+|C|\ge s+4\ge6>4=\delta+1.
\end{equation*}
If \(p=3\), then
$p\ge\sum_{u\in S}d_G(u)-3s+3$
and
$\sum_{u\in S}d_G(u)\ge3s$.
This implies that \(\sum_{u\in S}d_G(u)=3s\).
Hence
\begin{equation*}
c=\sum_{u\in S}d_G(u)-2s+2=s+2.
\end{equation*}
Thus \(c-p=s-1\).
Since \(p=3\), the remaining \(c-p=s-1\) components have \(r_i\ge3\).
By Lemma~\ref{lem:barrier}, \(r_i\le s\), and hence \(s\ge3\). Thus
\begin{equation*}
|R|\ge s+(c-p)=s+(s-1)\ge5>4=\delta+1.
\end{equation*}
\end{enumerate}

This proves \eqref{equ-10}, and hence every component of \(H\) has order at
most \(n-\delta^2-\delta\).
\end{proof}

By Claim~\ref{c-4.2}, we have
\begin{equation}\label{equ-11}
\rho(H)\le n-\delta(\delta+1)-1.
\end{equation}
By Lemma~\ref{lem-2}, together with \eqref{equ-9} and \eqref{equ-11}, we obtain
\begin{equation}\label{equ-12}
\rho(G)\le\rho(H)+\rho(F)\le n-\delta(\delta+1)-1+\sqrt{\delta}<n-\delta^2-1.
\end{equation}
Since $K_{n-\delta^2}\subset \mathcal{A}_{n,\delta}$, by Lemma~\ref{lem-1}, we have $\rho(\mathcal{A}_{n,\delta})>n-\delta^2-1$. Hence
\begin{equation*}
\rho(G)<\rho(\mathcal{A}_{n,\delta}),
\end{equation*}
as desired.
\end{proof}

\section{Proof of Theorem~\ref{thm:delta2}}

We first compare $\mathcal B_n$ with the natural bridge competitor.  Let $D_n$ be the graph obtained from a vertex $u$, a clique $K_{n-4}$ with root $a$, and a clique $K_3$ with root $b$, by adding the two bridge edges $ua$ and $ub$.

\begin{lemma}\label{lem:Ptheta}
For $n\ge13$, $\rho(D_n)<\rho(\mathcal B_n)$.
\end{lemma}

\begin{proof}
Let \(c,d\) be the two non-root vertices of the \(K_3\) in \(D_n\). 
Then \(\mathcal B_n=D_n-cd+ac+ad\). 
\begin{figure}[htbp]
\centering
\begin{tikzpicture}
[scale=0.95,
every node/.style={font=\small},
dot/.style={circle,fill=black,inner sep=1.7pt,outer sep=0pt},
bigclique/.style={
circle,
draw=black,
thick,
minimum size=2.25cm,
align=center},
edge/.style={thick}
]
\begin{scope}[xshift=-3.75cm]
\node[bigclique] (K1) at (-1.95,0) {\(K_{n-4}\)};
\node[dot] (a1) at (K1.0) {};
\node at (-0.55,-0.30) {\(a\)};
\node[dot] (u1) at (0.35,0) {};
\node at (0.35,-0.30) {\(u\)};
\node[dot] (b1) at (1.55,0) {};
\node at (1.55,-0.30) {\(b\)};
\node[dot] (c1) at (2.65,0.75) {};
\node at (2.65,1.05) {\(c\)};
\node[dot] (d1) at (2.65,-0.75) {};
\node at (2.65,-1.05) {\(d\)};
\draw[edge] (a1)--(u1);
\draw[edge] (u1)--(b1);
\draw[edge] (b1)--(c1);
\draw[edge] (b1)--(d1);
\draw[edge] (c1)--(d1);
\node at (2.25,0) {\(K_3\)};
\node[font=\bfseries] at (0.35,-1.75) {\(D_n\)};
\end{scope}
\draw[->,thick] (0,0.15)--(1,0.15);
\begin{scope}[xshift=4.35cm]
\node[bigclique] (K2) at (-1.35,0) {\(K_{n-4}\)};
\node[dot] (a2) at (K2.0) {};
\node at (-0.08,-0.30) {\(a\)};
\node[dot] (u2) at (0.95,0) {};
\node at (0.95,-0.30) {\(u\)};
\node[dot] (b2) at (2.15,0) {};
\node at (2.38,0) {\(b\)};
\node[dot] (c2) at (0.95,1.00) {};
\node at (0.95,1.30) {\(c\)};
\node[dot] (d2) at (0.95,-1.00) {};
\node at (0.95,-1.30) {\(d\)};
\draw[edge] (a2)--(u2)--(b2);
\draw[edge] (a2)--(c2)--(b2);
\draw[edge] (a2)--(d2)--(b2);
\node[font=\bfseries] at (0.55,-1.75) {\(\mathcal B_n\)};
\end{scope}
\end{tikzpicture}
\caption{The transformation from \(D_n\) to \(\mathcal B_n\).}
\label{fig:Dn-to-Bn}
\end{figure}
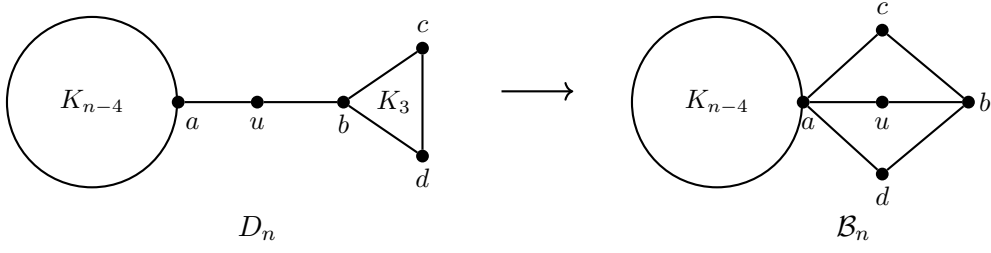
Let $\mb{x}$ be the positive Perron vector of $D_n$, and let $\lambda=\rho(D_n)$. 
By symmetry, \(x_c=x_d\). According to \(A(D_n)\mb{x}=\lambda \mb{x}\), we have 
\begin{equation*}
\lambda x_c=x_b+x_c,\qquad
\lambda x_b=x_u+2x_c,\qquad
\lambda x_u=x_a+x_b.
\end{equation*}
Thus, we obtain \(x_b=(\lambda-1)x_c\), \(x_u=(\lambda^2-\lambda-2)x_c\) and 
\(x_a=(\lambda^3-\lambda^2-3\lambda+1)x_c\). 

Note that \(K_{n-4}\) is a proper subgraph of \(D_n\) and \(n\geq 13\). Then Lemma~\ref{lem-1} 
implies that $\lambda>n-5\ge8$. Hence $\lambda^3-\lambda^2-3\lambda+1>1$ for \(\lambda\geq 8\). 
Thus, \(x_a>x_c\). Therefore
\[
\mb{x}^T(A(\mathcal B_n)-A(D_n))\mb{x}=4x_ax_c-2x_c^2=2x_c(2x_a-x_c)>0.
\]
Thus $\rho(\mathcal B_n)>\rho(D_n)$.
\end{proof}

\begin{lemma}\label{lem:coalescing}
Let \(N\ge9\). Form a graph from a clique \(K_N\), a vertex \(b\), and
three vertices \(x_1,x_2,x_3\) by joining \(b\) to every \(x_i\) and joining
each \(x_i\) to one vertex of the clique. Up to
isomorphism, the three possible graphs are shown in Figure~\ref{ffff-1}. 
Among all choices of the three
clique-neighbours of the \(x_i\), the spectral radius is maximized exactly
when all three \(x_i\) have the same clique-neighbour. In this case the
maximizing graph is \(\mathcal B_{N+4}\), up to isomorphism. 
\begin{figure}[htbp]
\centering
\begin{tikzpicture}
[scale=0.9,
every node/.style={font=\small},
dot/.style={circle,fill=black,inner sep=1.5pt,outer sep=0pt},
clique/.style={
circle,
draw=black,
thick,
minimum size=1.75cm,
align=center},
edge/.style={thick},
labelnode/.style={inner sep=1pt,fill=white}
]
\begin{scope}[xshift=-5.7cm]
\node[clique] (K1) at (-1.35,0) {\(K_N\)};
\node[dot] (a11) at (K1.35) {};
\node[dot] (a12) at (K1.0) {};
\node[dot] (a13) at (K1.-35) {};
\node[dot] (x11) at (0.70,0.78) {};
\node[labelnode] at (0.70,1.06) {\(x_1\)};
\node[dot] (x12) at (0.70,0) {};
\node[labelnode] at (0.70,-0.28) {\(x_2\)};
\node[dot] (x13) at (0.70,-0.78) {};
\node[labelnode] at (0.70,-1.06) {\(x_3\)};
\node[dot] (b1) at (1.90,0) {};
\node[labelnode] at (2.13,0) {\(b\)};
\draw[edge] (a11)--(x11)--(b1);
\draw[edge] (a12)--(x12)--(b1);
\draw[edge] (a13)--(x13)--(b1);
\end{scope}
\begin{scope}
\node[clique] (K2) at (-1.35,0) {\(K_N\)};
\node[dot] (a21) at (K2.25) {};
\node[dot] (a22) at (K2.-35) {};
\node[dot] (x21) at (0.70,0.78) {};
\node[labelnode] at (0.70,1.06) {\(x_1\)};
\node[dot] (x22) at (0.70,0) {};
\node[labelnode] at (0.70,-0.28) {\(x_2\)};
\node[dot] (x23) at (0.70,-0.78) {};
\node[labelnode] at (0.70,-1.06) {\(x_3\)};
\node[dot] (b2) at (1.90,0) {};
\node[labelnode] at (2.13,0) {\(b\)};
\draw[edge] (a21)--(x21)--(b2);
\draw[edge] (a21)--(x22)--(b2);
\draw[edge] (a22)--(x23)--(b2);
\end{scope}
\begin{scope}[xshift=5.7cm]
\node[clique] (K3) at (-1.35,0) {\(K_N\)};
\node[dot] (a31) at (K3.0) {};
\node[dot] (x31) at (0.70,0.78) {};
\node[labelnode] at (0.70,1.06) {\(x_1\)};
\node[dot] (x32) at (0.70,0) {};
\node[labelnode] at (0.70,-0.28) {\(x_2\)};
\node[dot] (x33) at (0.70,-0.78) {};
\node[labelnode] at (0.70,-1.06) {\(x_3\)};
\node[dot] (b3) at (1.90,0) {};
\node[labelnode] at (2.13,0) {\(b\)};
\draw[edge] (a31)--(x31)--(b3);
\draw[edge] (a31)--(x32)--(b3);
\draw[edge] (a31)--(x33)--(b3);
\end{scope}
\end{tikzpicture}
\caption{The three possible graphs up to isomorphism.}
\label{ffff-1}
\end{figure}
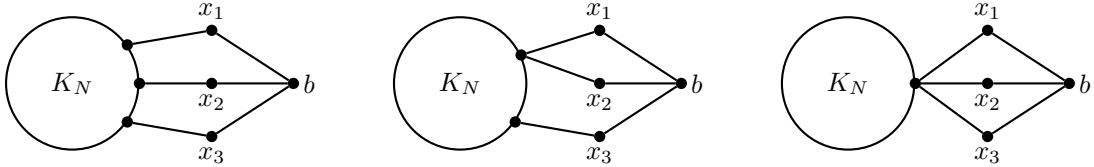
\end{lemma}

\begin{proof}
For \(v\in V(K_N)\), let \(k(v)=|\{i: vx_i\in E(G)\}|\). 
Then \(k(v)\in\{0,1,2,3\}\) and
\[
\sum_{v\in V(K_N)}k(v)=3.
\]
Thus, up to isomorphism, the possible nonzero multiplicity patterns are
\((3)\), \((2,1)\) and \((1,1,1)\).

Let \(\lambda=\rho(G)\), and let \(\mb{x}\) be the positive unit Perron
vector of \(G\). Since \(K_N\) is a proper subgraph of the connected graph
\(G\), Lemma~\ref{lem-1} implies that 
\[
\lambda>\rho(K_N)=N-1\ge8, 
\]
for \(N\geq 9\). 

Let \(v\in V(K_N)\). For every \(x_i\) adjacent to \(v\), we have
\(\lambda x_i=x_v+x_b\). Hence
\begin{equation*}
\lambda x_v
=
\sum_{z\in V(K_N)}x_z-x_v
+
k(v)\frac{x_v+x_b}{\lambda}.
\end{equation*}
Thus
\begin{equation*}
x_v=
\frac{\lambda\sum_{z\in V(K_N)}x_z+k(v)x_b}
{\lambda^2+\lambda-k(v)}.
\end{equation*}
Since \(\lambda\ge8\), it follows that \(\lambda^2+\lambda-k(v)>0\) for every 
\(0\le k(v)\le3\). For \(0\le q<p\le3\),
\[
\frac{\lambda\sum_{z\in V(K_N)}x_z+px_b}{\lambda^2+\lambda-p}
-
\frac{\lambda\sum_{z\in V(K_N)}x_z+qx_b}{\lambda^2+\lambda-q} 
=
\frac{(p-q)\left(\lambda\sum_{z\in V(K_N)}x_z+(\lambda^2+\lambda)x_b\right)}
{(\lambda^2+\lambda-p)(\lambda^2+\lambda-q)}
>0.
\]

Therefore, if \(v,w\in V(K_N)\) satisfy \(k(w)>k(v)\), then \(x_w>x_v\). 

Let \(v,w\in V(K_N)\) with \(vx_i\in E(G)\) and \(k(w)\ge k(v)\), and let
\(G'=G-vx_i+wx_i\). Then 
\begin{equation*}
\mb{x}^T(A(G')-A(G))\mb{x}=2x_i(x_w-x_v)\ge0.
\end{equation*}
By the Rayleigh quotient, \(\rho(G')\ge\rho(G)\). If \(k(w)>k(v)\), then
\(x_w>x_v\), and hence \(\rho(G')>\rho(G)\).

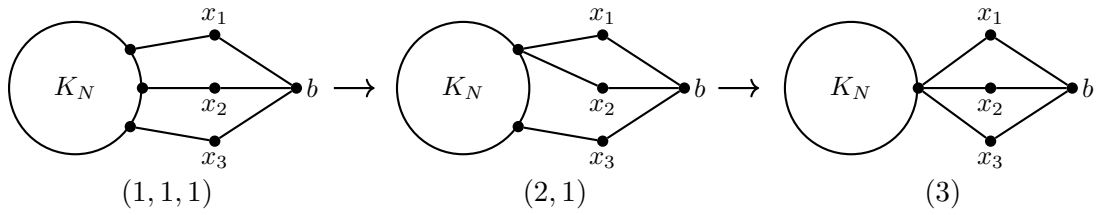
\begin{figure}[htbp]
\centering
\begin{tikzpicture}
[scale=0.9,
every node/.style={font=\small},
dot/.style={circle,fill=black,inner sep=1.5pt,outer sep=0pt},
clique/.style={
circle,
draw=black,
thick,
minimum size=1.75cm,
align=center},
edge/.style={thick},
labelnode/.style={inner sep=1pt,fill=white}
]
\begin{scope}[xshift=-5.7cm]
\node[clique] (K1) at (-1.35,0) {\(K_N\)};

\node[dot] (a11) at (K1.35) {};
\node[dot] (a12) at (K1.0) {};
\node[dot] (a13) at (K1.-35) {};
\node[dot] (x11) at (0.70,0.78) {};
\node[labelnode] at (0.70,1.06) {\(x_1\)};
\node[dot] (x12) at (0.70,0) {};
\node[labelnode] at (0.70,-0.28) {\(x_2\)};
\node[dot] (x13) at (0.70,-0.78) {};
\node[labelnode] at (0.70,-1.06) {\(x_3\)};
\node[dot] (b1) at (1.90,0) {};
\node[labelnode] at (2.13,0) {\(b\)};
\draw[edge] (a11)--(x11)--(b1);
\draw[edge] (a12)--(x12)--(b1);
\draw[edge] (a13)--(x13)--(b1);
\node[font=\bfseries] at (0,-1.55) {\((1,1,1)\)};
\end{scope}
\draw[->,thick] (-3.25,0)--(-2.65,0);
\begin{scope}
\node[clique] (K2) at (-1.35,0) {\(K_N\)};
\node[dot] (a21) at (K2.35) {};
\node[dot] (a22) at (K2.-35) {};
\node[dot] (x21) at (0.70,0.78) {};
\node[labelnode] at (0.70,1.06) {\(x_1\)};
\node[dot] (x22) at (0.70,0) {};
\node[labelnode] at (0.70,-0.28) {\(x_2\)};
\node[dot] (x23) at (0.70,-0.78) {};
\node[labelnode] at (0.70,-1.06) {\(x_3\)};
\node[dot] (b2) at (1.90,0) {};
\node[labelnode] at (2.13,0) {\(b\)};
\draw[edge] (a21)--(x21)--(b2);
\draw[edge] (a21)--(x22)--(b2);
\draw[edge] (a22)--(x23)--(b2);
\node[font=\bfseries] at (0,-1.55) {\((2,1)\)};
\end{scope}
\draw[->,thick] (2.4,0)--(3,0);
\begin{scope}[xshift=5.7cm]
\node[clique] (K3) at (-1.35,0) {\(K_N\)};
\node[dot] (a31) at (K3.0) {};
\node[dot] (x31) at (0.70,0.78) {};
\node[labelnode] at (0.70,1.06) {\(x_1\)};
\node[dot] (x32) at (0.70,0) {};
\node[labelnode] at (0.70,-0.28) {\(x_2\)};
\node[dot] (x33) at (0.70,-0.78) {};
\node[labelnode] at (0.70,-1.06) {\(x_3\)};
\node[dot] (b3) at (1.90,0) {};
\node[labelnode] at (2.13,0) {\(b\)};
\draw[edge] (a31)--(x31)--(b3);
\draw[edge] (a31)--(x32)--(b3);
\draw[edge] (a31)--(x33)--(b3);
\node[font=\bfseries] at (0,-1.55) {\((3)\)};
\end{scope}
\end{tikzpicture}
\caption{The three patterns. }
\label{three}
\end{figure}

Let \(G_{(1,1,1)}\), \(G_{(2,1)}\), and \(G_{(3)}\) be the three graphs shown
in Figure~\ref{three}. Then
\[
G_{(1,1,1)}\longrightarrow G_{(2,1)}\longrightarrow G_{(3)}
\]
by the edge-moving operation described above. 
The first move is from multiplicity \(1\) to multiplicity \(1\), while the
second move is from multiplicity \(1\) to multiplicity \(2\). Hence, by the
edge-moving argument above,
\[
\rho(G_{(1,1,1)})\le \rho(G_{(2,1)})<\rho(G_{(3)}).
\]
Therefore the maximum is attained only for the pattern \((3)\). Equivalently,
all three \(x_i\) have the same clique-neighbour, and the resulting graph is
\(\mathcal B_{N+4}\), up to isomorphism. 
\end{proof}

\begin{lemma}\label{lem:U-small}
Let $n\geq 13$ and $t\in\{5,6,7\}$. Let $U_{n,t}$ be the graph obtained
by identifying one vertex of a clique $K_{n-t}$ with one vertex of a
clique $K_{t+1}$. Then $\rho(U_{n,t})<n-5$.
\end{lemma}

\begin{proof}
Let $a$ be the common vertex of the two cliques. Then $A(U_{n,t})$ has
the equitable quotient matrix
\[
Q_t=
\begin{pmatrix}
0&n-t-1&t\\
1&n-t-2&0\\
1&0&t-1
\end{pmatrix}
\]
with respect to the partition \(V(G)=\{a\}\cup \{ V(K_{n-t})\setminus\{a\}\}\cup V(K_{t+1})\setminus\{a\}\). 

By a simple calculation, the characteristic polynomial of $Q_t$ is
\[
\phi_t(x)
=x^3+(3-n)x^2+(nt-t^2-t-2n+3)x
  +2nt-n-2t^2-2t+1 .
\]
Hence, by Lemma~\ref{lem:equitable}, $\rho(U_{n,t})$ is the largest root
of $\phi_t(x)=0$.

For $t=5,6,7$, we have
\[
\begin{array}{lll}
\phi_5(n-5)=n^2-13n+26,&
\phi_6(n-5)=2n^2-28n+62,&
\phi_7(n-5)=3n^2-45n+104,
\\[1mm]
\phi'_5(n-5)=n^2-11n+18,&
\phi'_6(n-5)=n^2-10n+6,&
\phi'_7(n-5)=n^2-9n-8.
\end{array}
\]
It is clear that all the above quantities are positive for $n\geq 13$.
Moreover, \(\phi_t''(x)=6x+6-2n\). 
Thus, for $x\geq n-5$,
\[
\phi_t''(x)\geq 6(n-5)+6-2n=4n-24>0.
\]
Therefore $\phi'_t(x)>0$ for $x\geq n-5$, and consequently
\[
\phi_t(x)\geq \phi_t(n-5)>0, 
\]
for \(x\geq n-5\). 
So $\phi_t(x)$ has no root in $[n-5,\infty)$. It follows that
\[
        \rho(U_{n,t})=\lambda_1(Q_t)<n-5,
\]
as desired.
\end{proof}

\begin{proof}[Proof of Theorem~\ref{thm:delta2}]
Let \(G\) be a connected graph of order \(n\ge 13\) with \(\delta(G)= 2\). Suppose
to the contrary that \(G\) has no even factor. Choose a barrier set \(S\). Let
\(|S|=s\geq 1\), and let  \(c\) be the number of components of \(G-S\).
Let \(C_0\) be a largest component of \(G-S\), and let \(t=n-|C_0|\).
Since \(\delta(G)= 2\), \(\sum_{u\in S}d_G(u)\geq 2|S|=2s\).
By Lemma~\ref{lem:barrier}, we obtain
\begin{equation}\label{equ-1}
        c=\sum_{u\in S}d_G(u)-2s+2\ge 2.
\end{equation}
Hence \(G-S\) has at least two components.
\begin{claim}\label{c-5.1}
$t\ge 4$.
\end{claim}

\begin{proof}
Indeed, let \(C\ne C_0\) be a component of \(G-S\). If \(e_G(C,S)=1\), then
Lemma~\ref{lem:one-edge-size} implies that \(|C|\ge 3\), and hence
\(t\ge |S|+|C|\ge 4\).
If \(e_G(C,S)\geq 3\), then Lemma~\ref{lem:barrier} implies \(e_G(C,S)\leq s\),
and hence \(s\geq 3\).  Since \(C\) is nonempty, \(t\geq s+|C|\geq 4\).
\end{proof}

Next we divide the proof into the following three subcases according to
the value of \(t\).

\n{\bf Case 1.1.} \(t=4\).

We first prove that \(G-S\) has exactly one component distinct from \(C_0\).
Since \(c\ge 2\), there exists at least one component of \(G-S\) different
from \(C_0\). Suppose that there are two such components, say \(C_1\) and
\(C_2\). By the proof of Claim~\ref{c-5.1}, applied to \(C_1\), we have
\(s+|C_1|\geq 4\).  Since \(C_2\) is nonempty, it follows that
\begin{equation*}
        t\ge s+|C_1|+|C_2|\ge 5,
\end{equation*}
which contradicts \(t=4\). Hence \(G-S\) has exactly one component distinct
from \(C_0\).

Let \(C\) be the unique component of \(G-S\) distinct from \(C_0\).
If \(e_G(C,S)=1\), then Lemma~\ref{lem:one-edge-size} implies that \(|C|\ge 3\). Since
\(t=s+|C|=4\), it follows that \(s=1\) and \(|C|=3\).

Write $S=\{u\}$.  Let $a$ be the unique neighbour of $u$ in $C_0$, and let $b$ be the unique neighbour of $u$ in $C$.  If we complete $C_0$ to a clique and complete $C$ to a clique, keeping only the two edges $ua$ and $ub$ between $u$ and the two components, then we obtain exactly the graph $D_n$.  Thus $G$ is a spanning subgraph of $D_n$.  By Lemma~\ref{lem-1} and Lemma~\ref{lem:Ptheta},
\[
        \rho(G)\leq\rho(D_n)<\rho(\mathcal B_n).
\]

It remains to consider the case \(e_G(C,S)\ge 3\). Since \(t=4\), we must have
\(s=3\) and \(|C|=1\). Let \(C=\{b\}\). Then \(b\) is adjacent to all vertices
of \(S\). By \eqref{equ-1}, we have
\begin{equation*}
        \sum_{u\in S}d_G(u)=c+2s-2=6.
\end{equation*}
Since each vertex of \(S\) has degree at least two and \(s=3\),
every vertex of \(S\) has degree exactly two. Thus each vertex of \(S\)
has one neighbour \(b\) and one neighbour in \(C_0\). Completing \(C_0\)
to a clique and applying Lemma~\ref{lem:coalescing}, we obtain
\begin{equation*}
        \rho(G)\le \rho(\mathcal{B}_n),
\end{equation*}
where equality holds if and only if \(G\cong \mathcal{B}_n\).

\n{\bf Case 1.2.} \(5\le t\le 7\).

Let \(G'\) be the graph obtained from \(G\) by completing \(C_0\) and
\(V(G)\setminus V(C_0)\) into cliques. Then \(G\) is a spanning subgraph of \(G'\).  Lemma~\ref{lem-1} gives
\begin{equation*}
        \rho(G)\leq \rho(G').
\end{equation*}
\begin{claim}
Every vertex in \(V(G')\setminus V(C_0)\) has at most one neighbour in \(C_0\).
\end{claim}
\begin{proof}
The graph \(G'\) is obtained from \(G\) by adding only edges inside \(C_0\) and
inside \(V(G)\setminus V(C_0)\).  Hence \(e_{G'}(z,C_0)=e_G(z,C_0)\) for every
\(z\in V(G)\setminus V(C_0)\).  Let \(z\in V(G)\setminus V(C_0)\).
If \(z\in S\), then Lemma~\ref{lem:barrier} implies that \(e_G(z,C_0)\leq 1\).
If \(z\notin S\), then \(z\) belongs to a component of \(G-S\) different from
\(C_0\), and hence \(e_G(z,C_0)=0\).
Therefore \(e_{G'}(z,C_0)\leq 1\).
\end{proof}

Let \(\mb{x}=(x_v)_{v\in V(G')}\) be the Perron vector of \(G'\). Choose a
vertex \(a\in C_0\) such that \(x_a=\max\{x_v:v\in V(C_0)\}\).
Let \(Z=\{z\in V(G')\setminus V(C_0): e_{G'}(z,C_0)=1\}\).
For each \(z\in Z\), let \(v_z\) be the unique neighbour of \(z\) in \(C_0\).
Define \(G''=G'-\{zv_z:z\in Z,\ v_z\ne a\}+\{za:z\in Z,\ v_z\ne a\}\).
Since $x_a\geq x_{v_z}$ for every moved edge, repeated applications of Lemma~\ref{lem-4} give \(\rho(G'')\geq \rho(G')\).  Adding all missing edges from \(a\) to
\(V(G')\setminus V(C_0)\) gives \(U_{n,t}\).  Therefore, Lemma~\ref{lem-1} gives
\begin{equation*}
        \rho(G')\leq \rho(G'')\leq \rho(U_{n,t}).
\end{equation*}
Combining this with \(\rho(G)\leq \rho(G')\), we obtain
\begin{equation}\label{equ-3}
\rho(G)\leq \rho(U_{n,t}).
\end{equation}
Since $K_{n-4}$ is a proper subgraph of $\mathcal{B}_n$, Lemma~\ref{lem-1} gives $\rho(\mathcal{B}_n)>n-5$.
By Lemma~\ref{lem:U-small} and \eqref{equ-3}, for \(5\leq t\leq 7\),
\(\rho(G)\leq \rho(U_{n,t})<n-5<\rho(\mathcal{B}_n)\).

\medskip
\n{\bf Case 1.3.} \(t\ge 8\).

Let \(F\) be the bipartite graph with \(E(F)=\{uv\in E(G):u\in S,\
v\in V(G)\setminus S\}\), and let \(H=G-E(F)\).

\begin{claim}
The largest component of \(H\) has order \(n-t\).
\end{claim}
\begin{proof}
The components of \(H\) are the components of \(G-S\) and the isolated
vertices of \(S\). Thus \(C_0\) is a component of \(H\), and every other
component of \(H\) has order at most \(|C_0|\). Therefore the largest
component of \(H\) has order \(|C_0|=n-t\).
This proves the claim.
\end{proof}

Thus
\begin{equation}\label{equ-4}
\rho(H)\le n-t-1.
\end{equation}

\begin{claim}
For the graph \(F\), we have
\begin{equation}\label{equ-5}
|E(F)|\le 2t-1.
\end{equation}
\end{claim}
\begin{proof}
Since \(S\) is independent and \(F\) consists of all edges between \(S\) and
\(G-S\), we have \(|E(F)|=\sum_{u\in S}d_G(u)\). By Lemma~\ref{lem:barrier},
\(c=\sum_{u\in S}d_G(u)-2s+2\), and hence
\(\sum_{u\in S}d_G(u)=c+2s-2\). Since the \(c-1\) components of \(G-S\)
distinct from \(C_0\) are contained in \(V(G)\setminus (S\cup V(C_0))\), we
have \(c-1\le t-s\). Also \(s\le t\). Thus
\[
|E(F)|=\sum_{u\in S}d_G(u)=c+2s-2\le(t-s+1)+2s-2=t+s-1\le2t-1.
\]
This proves the claim.
\end{proof}

By Lemmas~\ref{lem-2} and \ref{lem-3}, together with \eqref{equ-4} and \eqref{equ-5}, we obtain
\[
\rho(G)\le \rho(H)+\rho(F)\le n-t-1+\sqrt{|E(F)|}
\le n-t-1+\sqrt{2t-1}.
\]
For $t\geq8$,
\[
        (t-4)^2-(2t-1)=t^2-10t+17>0,
\]
because the polynomial $t^2-10t+17$ is increasing for $t\geq8$ and has value $1$ at $t=8$.  Hence $\sqrt{2t-1}<t-4$, and therefore
\begin{equation}\label{equ-6}
        \rho(G)<n-t-1+t-4=n-5.
\end{equation}
Recall that \(\rho(\mathcal{B}_n)>n-5\).  By \eqref{equ-6}, \(\rho(G)<\rho(\mathcal{B}_n)\), as desired.
\end{proof}

\section{Proof of Theorem~\ref{thm:connected-even}}

For integers $a,b\geq2$, let $T(a,b)$ be the graph obtained from two disjoint rooted cliques $K_a$ and $K_b$ by adding one edge between the two roots.  Thus $\mathcal C_{n,\delta}=T(\delta+1,n-\delta-1)$.

\begin{lemma}\label{lem:two-clique-compression}
Let $m\geq2$ and $n\geq2m$. Among the graphs $T(a,n-a)$ with
$m\leq a\leq n-m$, the spectral radius is maximized only by
$T(m,n-m)$ and $T(n-m,m)$.
\end{lemma}

\begin{proof}
Let $a,b$ be integers such that $m<a\leq b$ and $a+b=n$. We show that
\[
        \rho(T(a-1,b+1))>\rho(T(a,b)).
\]
Let $T=T(a,b)$, let $\lambda=\rho(T)$, and let $\mb{x}$ be the positive
unit Perron vector of $A(T)$. For the branch of order $s\in\{a,b\}$, let
$r_s$ be its root and let $W_s$ be the set of its non-root vertices. By
symmetry, \(x_{r_s}=y_s\) and \(x_v=z_s\), where \(v\in W_s\). By \(A(T)\mb{x}=\lambda\mb{x}\), we have 
\begin{equation}\label{e-a1}
\lambda y_s=y_{n-s}+(s-1)z_s,
\end{equation}
and
\begin{equation}\label{e-a2}
\lambda z_s=y_s+(s-2)z_s.
\end{equation}
Since $K_s$ is a proper subgraph of the connected graph $T$, Lemma~\ref{lem-1} implies that 
$\lambda>\rho(K_s)=s-1$. Hence $\lambda-s+2>0$, and by \eqref{e-a2},
\begin{equation}\label{e-a3}
        z_s=\frac{y_s}{\lambda-s+2}.
\end{equation}
Substituting \eqref{e-a3} into \eqref{e-a1}, we get \( y_{n-s}=\theta_sy_s\), where 
\[
        \theta_s=\lambda-\frac{s-1}{\lambda-s+2}.
\]
Then $y_b=\theta_a y_a$, $y_a=\theta_b y_b$, and so
$\theta_a\theta_b=1$. Moreover, since $b\geq a$,
\[
\frac{b-1}{\lambda-b+2}-\frac{a-1}{\lambda-a+2}
=
\frac{(b-a)(\lambda+1)}
{(\lambda-b+2)(\lambda-a+2)}
\geq0 .
\]
Thus $\theta_b\leq \theta_a$. Together with $\theta_a\theta_b=1$, this
implies $\theta_a\geq1\geq\theta_b$, and therefore $y_b\geq y_a$. 
On the other hand, by \eqref{e-a3}, 
\[
        \frac{z_b}{z_a}
        =\frac{y_b}{y_a}\cdot
        \frac{\lambda-a+2}{\lambda-b+2}
        \geq1,
\]
so $z_b\geq z_a$.

Choose $w\in W_a$, and let $T'=T(a-1,b+1)$ be obtained from $T$ by moving
$w$ from the branch of order $a$ to the branch of order $b$. Let 
\[
        \Delta^-=\{wv:\ v\in \{r_a\}\cup(W_a\setminus\{w\})\},
        \qquad
        \Delta^+=\{wv:\ v\in \{r_b\}\cup W_b\}.
\]
Then $E(T')=(E(T)\setminus\Delta^-)\cup\Delta^+$. 
Note that $x_w=z_a$. Then 
\begin{align*}
\mb{x}^{T}(A(T')-A(T))\mb{x}
&=2z_a\left(
        \sum_{v\in\{r_b\}\cup W_b}x_v
        -
        \sum_{v\in\{r_a\}\cup(W_a\setminus\{w\})}x_v
        \right)                                                     \\
&=2z_a\bigl(y_b+(b-1)z_b-y_a-(a-2)z_a\bigr)                         \\
&\geq 2z_a\bigl(y_a+(b-1)z_a-y_a-(a-2)z_a\bigr)                     \\
&=2(b-a+1)z_a^2\\&>0 .
\end{align*}
Consequently, \( \rho(T')>\rho(T)\). 

Repeating this operation until the smaller branch has order $m$ gives the
desired maximum. Since the inequality is strict at each step, equality can
occur only for $T(m,n-m)$ and $T(n-m,m)$. 
\end{proof}

\begin{lemma}\label{lem:bridge-case}
Let \(G\) be a connected graph of order \(n\) with
\(\delta(G)\ge\delta\ge2\), where \(n\ge2\delta+2\). If \(G\) has a bridge,
then
\begin{equation*}
\rho(G)\le\rho(\mathcal C_{n,\delta}),
\end{equation*}
with equality if and only if \(G\cong\mathcal C_{n,\delta}\).
\end{lemma}

\begin{proof}
Let \(e\) be a bridge of \(G\), and let \(X\) and \(Y\) be the two components
of \(G-e\). Let \(|X|=a\) and \(|Y|=n-a\). Note that \(e\) is a bridge Then 
\(e_G(X,Y)=1\). 

We first show that
\begin{equation*}
a\ge\delta+1,\qquad n-a\ge\delta+1.
\end{equation*}
Suppose that \(1\le a\le\delta\). Then 
\begin{equation*}
1=e_G(X,Y)
=\sum_{v\in X}d_G(v)-2e_G(X)
\ge a\delta-a(a-1)
=a(\delta-a+1)\ge\delta,
\end{equation*}
which is impossible since \(\delta\ge2\). Hence \(a\ge\delta+1\). 
Similarly, we obtain \(n-a\ge\delta+1\). 

Now complete \(X\) and \(Y\) into cliques and keep the bridge \(e\). The
resulting graph is \(T(a,n-a)\), and \(G\) is a spanning subgraph of
\(T(a,n-a)\). Thus, by Lemma~\ref{lem-1}, we have \(\rho(G)\le\rho(T(a,n-a))\). 
Note that \(a\ge\delta+1\) and \(n-a\ge\delta+1\). 
By Lemma~\ref{lem:two-clique-compression}, we have 
\begin{equation*}
\rho(T(a,n-a))
\le
\rho(T(\delta+1,n-\delta-1))
=
\rho(\mathcal C_{n,\delta}).
\end{equation*}
Therefore 
\begin{equation*}
\rho(G)\le\rho(\mathcal C_{n,\delta}).
\end{equation*}

Equality can occur only when \(G=T(a,n-a)\) and \(\{a,n-a\}=\{\delta+1,n-\delta-1\}\), 
that is, \(G\cong\mathcal C_{n,\delta}\). Conversely, if
\(G\cong\mathcal C_{n,\delta}\), then equality holds. 
\end{proof}

\begin{lemma}\label{lem:two-trees-even}
If a graph \(G\) has two edge-disjoint spanning trees, then \(G\) has a
connected even factor.
\end{lemma}

\begin{proof}
Let \(T_1\) and \(T_2\) be two edge-disjoint spanning trees of \(G\). Set
\[
O=\{v\in V(G):d_{T_1}(v)\equiv1\pmod2\}.
\]
By the handshaking lemma, \(|O|\) is even.

We shall find a subgraph \(J\) of \(T_2\) such that
\[
d_J(v)\equiv
\begin{cases}
1\pmod2, & v\in O,\\
0\pmod2, & v\notin O.
\end{cases}
\]
Root \(T_2\) at a vertex \(r\). For each edge \(e=xy\in E(T_2)\), where
\(y\) is the child of \(x\), let \(S_y\) be the vertex set of the subtree
rooted at \(y\). Define
\[
J=\{xy\in E(T_2): |S_y\cap O|\equiv1\pmod2\}.
\]

Let \(C(v)\) be the set of children of \(v\). If \(v\ne r\), let \(e_v\)
be the edge joining \(v\) to its parent. Then, by the definition of \(J\),
\[
\begin{aligned}
d_J(v)
&\equiv |S_v\cap O|+\sum_{y\in C(v)}|S_y\cap O|                    \\
&\equiv |S_v\cap O|+|(S_v\setminus\{v\})\cap O|                    \\
&\equiv
\begin{cases}
1\pmod2, & v\in O,\\
0\pmod2, & v\notin O.
\end{cases}
\end{aligned}
\]
For the root \(r\), there is no parent edge, and hence
\[
d_J(r)
\equiv \sum_{y\in C(r)}|S_y\cap O|
=|O\setminus\{r\}|
\equiv
\begin{cases}
1\pmod2, & r\in O,\\
0\pmod2, & r\notin O,
\end{cases}
\]
since \(|O|\) is even.

Now let
\[
F=T_1\cup J.
\]
Since \(T_1\) is a spanning tree, \(F\) is spanning and connected. Moreover,
for every \(v\in V(G)\),
\[
d_F(v)=d_{T_1}(v)+d_J(v)\equiv0\pmod2.
\]
Thus \(F\) is a connected even factor of \(G\).
\end{proof}

\begin{lemma}[\cite{NashWilliams1961,Tutte1961}]\label{lem:nwt}
A graph $G$ has two edge-disjoint spanning trees if and only if, for every partition $\mathcal P$ of $V(G)$,
\[
        e_G(\mathcal P)\geq2(|\mathcal P|-1),
\]
where $e_G(\mathcal P)$ denotes the number of edges whose ends lie in different parts of $\mathcal P$.
\end{lemma}

\begin{lemma}[\cite{HongShuFang2001}]\label{lem:hsf}
Let $G$ be a connected graph of order $n$ and size $m$, and suppose that $\delta(G)\geq\delta\geq1$.  Then
\[
        \rho(G)\leq
        \frac{\delta-1}{2}+
        \sqrt{2m-n\delta+\frac{(\delta+1)^2}{4}}.
\]
Moreover, for nonnegative integers $p$ and $q$ with $2q\leq p(p-1)$ and $0\leq x\leq p-1$, 
the function $f(x)=(x-1)/2+\sqrt{2q-px+(1+x)^2/4}$ is decreasing with respect to $x$.
\end{lemma}

\begin{lemma}[\cite{Chen1993}]\label{lem:cai}
Let $G$ be a $2$-edge-connected simple graph of order $n$.  If
\[
        e(G)\geq \binom{n-4}{2}+6,
\]
then one of the following holds:
\begin{enumerate}[(i)]
\item $G$ is supereulerian;
\item $G\cong K_{2,5}$;
\item $e(G)=\binom{n-4}{2}+6$, and either $G\cong Q_3-v$, or $G$ contains a complete subgraph $H\cong K_{n-4}$ such that $G/H\cong K_{2,3}$.
\end{enumerate}
Here $Q_3-v$ denotes the cube with one vertex deleted, and $G/H$ denotes the graph obtained from $G$ by contracting $H$ to a single vertex and deleting loops.
\end{lemma}

\begin{lemma}\label{lem:low-degree-core-three}
Let $G$ be a $2$-edge-connected graph of order $n\geq8$ with $\delta(G)\geq3$.  If
\[
        \rho(G)\geq\rho(\mathcal C_{n,3}),
\]
then $G$ has a connected even factor.
\end{lemma}

\begin{proof}
Note that $K_{n-4}$ is a proper subgraph of $\mathcal C_{n,3}$. Then Lemma~\ref{lem-1} 
implies that 
\[
\rho(\mathcal C_{n,3})>\rho(K_{n-4})=n-5.
\]
Combining this with the assumption, we have $\rho(G)>n-5$. 
By Lemma~\ref{lem:hsf} with $\delta=3$, we obtain
\[
        n-5<\rho(G)\leq1+\sqrt{2e(G)-3n+4}.
\]
It follows that 
\[
        e(G)\geq\binom{n-4}{2}+7. 
\]

Assume, to the contrary, that \(G\) has no connected even factor. Then \(G\)
is not supereulerian. By Lemma~\ref{lem:cai}, one of the alternatives
\textup{(ii)} and \textup{(iii)} must occur.

The alternative \textup{(ii)} is impossible, since
\(\delta(K_{2,5})=2<3\le\delta(G)\). The alternative \textup{(iii)} is also
impossible, because it requires
\(e(G)=\binom{n-4}{2}+6\), whereas
\(e(G)\ge\binom{n-4}{2}+7\). This contradiction proves the lemma. 
\end{proof}

\begin{lemma}\label{lem:low-degree-core-two}
Let \(G\) be a \(2\)-edge-connected graph of order \(n\ge8\) with
\(\delta(G)\ge2\). If
\[
\rho(G)\ge\rho(\mathcal C_{n,2}),
\]
then \(G\) has a connected even factor.
\end{lemma}

\begin{proof}
Note that \(K_{n-3}\) is a proper subgraph of \(\mathcal C_{n,2}\). Then
Lemma~\ref{lem-1} implies that
\[
\rho(\mathcal C_{n,2})>\rho(K_{n-3})=n-4.
\]
Combining this with the assumption, we have \(\rho(G)>n-4\). By
Lemma~\ref{lem:hsf} with \(\delta=2\), we obtain
\[
n-4<\rho(G)\le \frac{1}{2}+\sqrt{2e(G)-2n+\frac{9}{4}}.
\]
It follows that
\[
e(G)\ge \binom{n-3}{2}+4.
\]
Moreover, 
\[
\left(\binom{n-3}{2}+4\right)-\left(\binom{n-4}{2}+6\right)=n-6>0.
\]
Then, we have
\[
e(G)>\binom{n-4}{2}+6.
\]

Assume, to the contrary, that \(G\) has no connected even factor. Then \(G\)
is not supereulerian. By Lemma~\ref{lem:cai}, one of the alternatives
\textup{(ii)} and \textup{(iii)} must occur.

The alternative \textup{(ii)} is impossible, since \(K_{2,5}\) has order
\(7\), whereas \(n\ge8\). The alternative \textup{(iii)} is also impossible,
because it requires \(e(G)=\binom{n-4}{2}+6\), whereas
\(e(G)>\binom{n-4}{2}+6\). This contradiction proves the lemma.
\end{proof}

\begin{proof}[Proof of Theorem~\ref{thm:connected-even}]
Suppose that \(G\) has no connected even factor. We show that
\(G\cong\mathcal C_{n,\delta}\).

If \(G\) has a bridge, then Lemma~\ref{lem:bridge-case} gives
\(\rho(G)\le \rho(\mathcal C_{n,\delta})\), with equality if and only if
\(G\cong\mathcal C_{n,\delta}\). Since
\(\rho(G)\ge \rho(\mathcal C_{n,\delta})\), it follows that
\(G\cong\mathcal C_{n,\delta}\). Hence we may assume that \(G\) has no
bridge. Then \(G\) is \(2\)-edge-connected. We now consider the following two cases.

\medskip
\noindent{\bf Case 1.} \(\delta\ge4\).

By Lemma~\ref{lem:two-trees-even}, \(G\) has no two edge-disjoint spanning
trees. Thus Lemma~\ref{lem:nwt} implies that there exists a partition
\(\mathcal P=\{V_1,V_2,\ldots,V_r\}\) of \(V(G)\) such that
\(q:=e_G(\mathcal P)\le2r-3\). Since \(G\) is bridgeless, every nonempty
proper vertex set sends at least two edges to its complement. Hence, if
\(c_i=e_G(V_i,V(G)\setminus V_i)\), then \(c_i\ge2\) for every \(i\). If
\(r=2\), then \(q=c_1=c_2\le1\), a contradiction. Thus \(r\ge3\).

Relabel the parts so that \(|V_1|\) is maximum. Let
\[
\mathcal I=\{i\in\{2,\ldots,r\}: |V_i|\le\delta\}.
\]
For \(i\in\mathcal I\), let \(|V_i|=s_i\). Then
\[
c_i=\sum_{v\in V_i}d_G(v)-2e_G(V_i)
\ge s_i\delta-s_i(s_i-1)
=s_i(\delta-s_i+1)
\ge\delta.
\]
Note that each crossing edge is counted twice in \(\sum_i c_i\). Then we have
\[
4r-6\ge2q=\sum_{i=1}^r c_i
\ge2+|\mathcal I|\delta+2(r-1-|\mathcal I|).
\]
Thus
\[
|\mathcal I|(\delta-2)\le2r-6.
\]
Note that \(\delta\ge4\). Then \(|\mathcal I|\le r-3\). Therefore at least
two of the parts \(V_2,\ldots,V_r\) have order at least \(\delta+1\), and so
\[
\sum_{i=2}^r |V_i|
\ge |\mathcal I|+(r-1-|\mathcal I|)(\delta+1)
\ge r-1+2\delta.
\]

Let \(F\) be the graph consisting of the \(q\) crossing edges of
\(\mathcal P\), and let \(H=G-E(F)\). By Lemma~\ref{lem-2}, we have 
\(\rho(G)\le\rho(H)+\rho(F)\). Since \(H\) is the disjoint union of
\(G[V_1],\ldots,G[V_r]\), it follows that 
\[
\rho(H)\le |V_1|-1
=n-\sum_{i=2}^r |V_i|-1
\le n-r-2\delta.
\]
By Lemma~\ref{lem:hong-component}, 
\(\rho(F)\le\sqrt{2q-1}\le\sqrt{4r-7}\). Therefore 
\[
\rho(G)\le n-r-2\delta+\sqrt{4r-7}.
\]
For \(r\ge3\) and \(\delta\ge4\), we have \(\sqrt{4r-7}<r+\delta-2\), and hence 
\[
\rho(G)<n-r-2\delta+r+\delta-2=n-\delta-2.
\]
On the other hand, \(K_{n-\delta-1}\) is a proper subgraph of
\(\mathcal C_{n,\delta}\). By Lemma~\ref{lem-1}, we have 
\(\rho(\mathcal C_{n,\delta})>\rho(K_{n-\delta-1})=n-\delta-2\). Hence
\(\rho(G)<\rho(\mathcal C_{n,\delta})\), contradicting the assumption.

\medskip
\noindent{\bf Case 2.} \(\delta\in\{2,3\}\).

If \(\delta=3\), then Lemma~\ref{lem:low-degree-core-three} implies that
\(G\) has a connected even factor, a contradiction. If \(\delta=2\), then
Lemma~\ref{lem:low-degree-core-two} gives the same contradiction.

Hence \(G\)  has a bridge.  Moreover, by Lemma~\ref{lem:bridge-case}, we have 
\(G\cong\mathcal C_{n,\delta}\). This completes the proof of Theorem~\ref{thm:connected-even}. 
\end{proof}

\vspace{6mm}

\n\textbf{Data Availability}  Data sharing not applicable to this
article as no datasets were generated or analysed during the current
study.

\vspace{2mm}

\n\textbf{Declarations}

\vspace{1mm}

\n\textbf{Conflict of interest}  The authors have no relevant
financial or non-financial interests to disclose.

\end{document}